\documentclass[11pt,a4paper]{article}
\usepackage[T1]{fontenc}
\usepackage[utf8]{inputenc}
\usepackage{amsmath,amssymb,amsthm,mathtools}
\usepackage{booktabs}
\usepackage{graphicx}
\usepackage{array}
\usepackage[margin=2.7cm]{geometry}
\usepackage[protrusion=true,expansion=false]{microtype}
\usepackage[hidelinks]{hyperref}

\theoremstyle{plain}
\newtheorem{theorem}{Theorem}[section]
\newtheorem{proposition}[theorem]{Proposition}
\newtheorem{lemma}[theorem]{Lemma}
\newtheorem{corollary}[theorem]{Corollary}
\newtheorem{conjecture}[theorem]{Conjecture}
\theoremstyle{definition}

\theoremstyle{remark}
\newtheorem{remark}[theorem]{Remark}
\newtheorem{question}[theorem]{Question}

\newtheorem*{theoremA}{Theorem A}
\newtheorem*{theoremB}{Theorem B}
\newtheorem*{theoremC}{Theorem C}
\newtheorem*{theoremD}{Theorem D}

\DeclareMathOperator{\ML}{ML}
\newcommand{\Z}{\mathbb{Z}}
\newcommand{\Q}{\mathbb{Q}}
\newcommand{\R}{\mathbb{R}}
\newcommand{\T}{\mathbb{T}}
\newcommand{\dist}[1]{\lVert #1\rVert}
\newcommand{\Dd}{D}
\newcommand{\UA}{U_A}
\newcommand{\UB}{U_B}
\newcommand{\UC}{U_C}

\numberwithin{equation}{section}

\title{Odd denominators in the Lonely Runner spectrum\\ for six speeds}
\author{Francesco Cordella\thanks{ENEA, Centro Ricerche Frascati, Via Enrico Fermi 45, 00044 Frascati (Rome), Italy. \texttt{francesco.cordella@enea.it}}}
\date{September 2026}

\begin{document}
\maketitle

\begin{abstract}
For distinct positive integers $v_1,\dots,v_n$ let $\ML(v_1,\dots,v_n)$ be the largest number $L$ such that at some time $t$ every $tv_i$ is at distance at least $L$ from the nearest integer; the Lonely Runner Conjecture asserts that $\ML\ge1/(n+1)$. Write $\ML=p/q$ in lowest terms. Kravitz conjectured that whenever $\ML<1/n$ one has $q=np+1$; Fan and Sun found counterexamples for $n=4$, conjectured that $q=np+k$ with $1\le k\le n$ always holds, and observed that in their data for $n=6$ only $k=1$ and $k=3$ occur. We explain this observation. For six speeds we show that all but finitely many tuples with $\ML<1/6$ satisfy $\ML=(P-1)/(6P)$ for an integer $P\equiv\pm1\pmod 6$; in particular $k\in\{1,3\}$ and the denominator $q$ is odd. The proof determines the three infinite two-parameter families of tuples on which such values concentrate, computes $\ML$ exactly on each family, and describes exactly where $k=3$ occurs. An exhaustive search over the $2\cdot10^9$ sextuples with speeds at most $110$ finds no exception. For five speeds the same method, together with Chen's classification of the tuples attaining $\ML=1/5$, shows that all but finitely many tuples with $\ML<1/5$ satisfy Kravitz's original conjecture. The computations are exact and the code is provided.
\end{abstract}

\noindent\textbf{MSC 2020:} 11K60 (primary); 11J13, 52C07, 11Y16.\\
\textbf{Keywords:} Lonely Runner Conjecture, maximum loneliness, Lonely Runner spectrum, view-obstruction, subtori.

\section{Introduction}\label{sec:intro}

\subsection{Lonely runners and their spectrum}
Suppose that $n+1$ runners start together on a circular track of unit length and run at pairwise distinct constant speeds. A runner is \emph{lonely} at a time when its distance from every other runner is at least $1/(n+1)$. The Lonely Runner Conjecture of Wills \cite{Wills} and Cusick \cite{Cusick74} asserts that every runner is lonely at some time. Passing to the frame of one runner, and using the reduction of Bohman, Holzman and Kleitman \cite{BHK} to integer speeds, the conjecture becomes the statement that for all distinct positive integers $v_1,\dots,v_n$
\[
\ML(v_1,\dots,v_n):=\max_{t\in\R}\ \min_{1\le i\le n}\dist{tv_i}\ \ge\ \frac1{n+1},
\]
where $\dist{x}$ denotes the distance from $x$ to the nearest integer. The conjecture is known for $n\le 9$: the cases $n\le 6$ by \cite{BetkeWills,CusickPomerance,BGGST,BHK,BarajasSerra08}, the case $n=7$ by Rosenfeld \cite{Rosenfeld} and the cases $n=8,9$ by Trakulthongchai \cite{Trak}, the last two with computer assistance. We refer to the survey of Perarnau and Serra \cite{PerarnauSerra} for the history of the problem.

Kravitz \cite{Kravitz} proposed to study not only the minimum but the whole set $S(n)$ of values of $\ML$, and conjectured that the part of $S(n)$ lying below $1/n$ is discrete, namely
\begin{equation}\label{eq:kravitz}
S(n)\cap(0,1/n)=\Bigl\{\frac{s}{ns+1}: s\in\mathbb N\Bigr\}.
\end{equation}
The values $s/(ns+1)$ are attained by the speeds $1,2,\dots,n-1,ns$. Kravitz proved \eqref{eq:kravitz} for $n\le 3$ and, for $n\in\{4,6\}$, in the regime where one speed is much larger than the others. Fan and Sun \cite{FanSun} found the family $\ML(8,4r+3,4r+11,4r+19)=(2r+7)/(4(2r+7)+2)$ of counterexamples for $n=4$ and proposed the amended conjecture that every value of $\ML$ below $1/n$ has the form $s/(ns+k)$ with $1\le k\le n$. Tabulating the values of $k$ that occur in their experiments, they noticed that for $n=6$ only $k\in\{1,3\}$ appears, and remarked that they could not explain why $k=2$ does not.

We call an $n$-tuple of distinct positive integers \emph{near-tight} if $\ML<1/n$, we always write $\ML=p/q$ in lowest terms, and we put $k:=q-np$. In this notation the amended conjecture of Fan and Sun says that $1\le k\le n$ for every near-tight tuple.

\subsection{Results}
The structural tool is the theory of relative Lonely Runner spectra of Giri and Kravitz \cite{GiriKravitz} and Jain and Kravitz \cite{JainKravitz}. In the view-obstruction formulation, an $n$-tuple of speeds is a one-dimensional subtorus $T=\langle (v_1,\dots,v_n)\rangle_\R$ of $\T^n:=(\R/\Z)^n$, and $\ML$ is expressed through the $L^\infty$-distance $\Dd(T)=1/2-\ML$ from $T$ to the centre $(1/2,\dots,1/2)$. The results of \cite{GiriKravitz} imply that all but finitely many near-tight $n$-tuples lie on a two-dimensional subtorus $U$ with $\Dd(U)=1/2-1/n$ (Lemma \ref{lem:reduction} below), and \cite{JainKravitz} describes the values of $\Dd$ on the one-dimensional subtori of such a $U$ in terms of finitely many arithmetic progressions. Our first result identifies the relevant subtori for six speeds.

\begin{theoremA}[Theorem \ref{thm:three}]
Up to permutations and sign changes of the coordinates, the two-dimensional proper subtori $U\subseteq\T^6$ with $\Dd(U)=1/3$ are exactly
\begin{gather*}
\UA=\langle(1,2,3,4,5,0),e_6\rangle_\R,\qquad \UB=\langle(1,3,4,5,9,0),e_6\rangle_\R,\\
\UC=\langle(1,0,1,2,3,3),(0,1,1,1,1,2)\rangle_\R .
\end{gather*}
\end{theoremA}

The first two are the subtori of the form $\langle V\rangle_\R\times\R/\Z$ over the two tight quintuples of Bohman, Holzman and Kleitman, that is, the ``one very fast runner'' subtori studied in \cite[Sections 8--10]{Kravitz}; the third is the subtorus $U^7$ of \cite[Section 6]{JainKravitz}, where it was exhibited as the source of a new infinite family of values, and where the authors write that computing its full relative spectrum ``would be fairly lengthy''. Theorem A is proved without any a priori bound on the entries of the generators; the finiteness comes from a projective argument (Lemmas \ref{lem:count}--\ref{lem:onetwo}) combined with the classification of the tight quintuples.

The second result determines $\ML$ on the three subtori. For a two-dimensional subtorus $U=\langle u,v\rangle_\R$ whose generators span the lattice $U\cap\Z^n$, the one-dimensional subtori of $U$ are the $T=\langle Au+Bv\rangle_\R$ with $A,B$ coprime, and we call $(A,B)$ the parameters of $T$.

\begin{theoremB}[Theorems \ref{thm:fast} and \ref{thm:UC}]
Let $T$ be a one-dimensional subtorus of $\UA$, $\UB$ or $\UC$ whose six speeds are nonzero and pairwise distinct in absolute value, and suppose $\ML(T)<1/6$. Then
\[
\ML(T)=\frac{P-1}{6P}\qquad\text{for an integer } P\equiv\pm1\pmod 6 .
\]
Equivalently, $k\in\{1,3\}$ and the denominator $q$ of $\ML(T)$ is odd. On $\UA$ and $\UB$ one always has $k=1$. On $\UC$, with parameters $(A,B)$ normalised by $A>0$, one has $k=3$ exactly when
\begin{center}
\begin{tabular}{@{}lll@{}}
\toprule
$(A,B)\bmod 6$ & condition & in terms of the slope $B/A$\\
\midrule
$(0,5)$ & $A+2B>0$ & $B/A>-1/2$\\
$(0,1)$ & $5A+2B<0$ & $B/A<-5/2$\\
$(5,1)$ & $5A+3B>0$ and $A-B>0$ & $-5/3<B/A<1$\\
$(5,2)$ & $4A+B>0$ and $4A+3B<0$ & $-4<B/A<-4/3$\\
\bottomrule
\end{tabular}
\end{center}
and $k=1$ in all other near-tight cases. Moreover $\max_i|v_i|<3q$ for every such $T\subseteq\UC$, and the constant $3$ is optimal.
\end{theoremB}

The progression $1/3+\tfrac16\mathrm{Prog}(6,11)$ of \cite[Theorem 1.5]{JainKravitz}, which corresponds to $A=6$ and $B\equiv5\pmod 6$, is contained in the first row of the table; the other three rows, and the fact that nothing else occurs, are new. The proof of the statement about $\UC$ is a complete, explicit version of the sector analysis of \cite{JainKravitz}: we make the thresholds in their Lemma 2.5 and Proposition 2.6 explicit, show that the relevant offsets depend only on $(A,B)$ modulo $6$, decompose the parameter plane into finitely many cones on which the competing candidates are homogeneous linear forms, treat separately the finitely many lattice lines on which some candidate has a bounded denominator, and check the remaining $23\,214$ parameter pairs directly. Everything is done in exact rational arithmetic.

Combining Theorems A and B with the reduction lemma gives the explanation of the observation of Fan and Sun.

\begin{theoremC}[Theorem \ref{thm:six}]
There is a finite set $E_6$ of near-tight sextuples such that every near-tight sextuple not in $E_6$ has $\ML=(P-1)/(6P)$ with $P\equiv\pm1\pmod 6$; in particular $k\in\{1,3\}$ and $q$ is odd. No sextuple with all speeds at most $110$ violates this conclusion. Among the $348$ near-tight primitive sextuples with speeds at most $110$, all $33$ with $k=3$ lie on $\UC$, and exactly nine lie on none of $\UA,\UB,\UC$; they are listed in Table \ref{tab:sporadic6} and all have $k=1$.
\end{theoremC}

The exceptional set $E_6$ is nonempty: the nine sporadic sextuples of Table \ref{tab:sporadic6} belong to it. We do not know an explicit bound for $E_6$; the constants in \cite{GiriKravitz} are effective in principle but far too large to be used, and we state the absence of a counterexample below $110$ as a conjecture (Conjecture \ref{conj:parity}).

For five speeds the same programme goes through with a simpler outcome. Jain and Kravitz remark that their method requires the classification of the tight instances with $n-1$ speeds, which they list as available for $n-1\in\{1,2,3,5\}$; in fact the case of four speeds is settled by Chen \cite{Chen91} (see also \cite{BarajasSerra09}), the tight quadruples being $(1,2,3,4)$ and $(1,3,4,7)$ up to dilation.

\begin{theoremD}[Theorems \ref{thm:two} and \ref{thm:five}]
Up to permutations and sign changes of the coordinates, the two-dimensional proper subtori $U\subseteq\T^5$ with $\Dd(U)=3/10$ are exactly $\langle(1,2,3,4,0),e_5\rangle_\R$ and $\langle(1,3,4,7,0),e_5\rangle_\R$. Consequently there is a finite set $E_5$ of near-tight quintuples such that every near-tight quintuple not in $E_5$ has $\ML=s/(5s+1)$ for some $s\in\mathbb N$, that is, satisfies \eqref{eq:kravitz}. No quintuple with speeds at most $130$ violates this conclusion.
\end{theoremD}

Thus for five speeds there is no analogue of $\UC$: both critical subtori are of one-very-fast-runner type, and the arithmetic of the relative spectrum is the one already worked out by Kravitz. In the language of \cite{JainKravitz}, $S_1(5)\cap(3/10,1/2]$ has finite symmetric difference with $3/10+\tfrac15\mathrm{Prog}(5,6)$; this is the analogue for five speeds of their Theorem 1.3.

\subsection{Methods, and what is computed}
The proofs have three ingredients. The classification of the critical subtori (Section \ref{sec:class}) is a finite argument: a zero-free integer vector of $U$ with a repeated entry must be a tight instance (Lemma \ref{lem:shape}), a counting argument on the projective line shows that $U$ always contains two independent such vectors (Lemmas \ref{lem:count}--\ref{lem:onetwo}), and an exact enumeration of the pairs finishes the job. The one-very-fast-runner subtori (Section \ref{sec:fast}) are handled by a short ``pre-jump'' argument, a direct computation up to an explicit bound, and Lemma 9.4 of \cite{Kravitz} beyond it. The subtorus $\UC$ (Section \ref{sec:UC}) requires the sector analysis of \cite{JainKravitz}; we describe it in enough detail that the reader can follow which finite computation is being performed and why its output is a proof.

The logical shape of the paper is thus: an empirical anomaly (no even $k$ for six speeds) is reduced to three critical subtori; on each of them an exact decomposition of the parameter plane identifies the value of $\ML$ as a linear form on every piece; reading the residue of that form modulo $6$ explains the anomaly and shows that it is only part of a stronger congruence, $P\equiv\pm1\pmod 6$; and what remains open is confined to the finitely many sporadic tuples that escape the reduction.

All computations are exact. The maximum loneliness of an integer tuple is computed by scanning the finitely many candidate times at which two runners coincide or are antipodal (Lemma \ref{lem:pairtimes}), with fractions compared by cross-multiplication in integer arithmetic; two independent implementations, in Python and in C, agree on every tuple tested. The accompanying code reproduces every number quoted in the paper; see Section \ref{sec:code}.

\subsection{Organisation}
Section \ref{sec:prelim} fixes notation, states the results from \cite{Kravitz,GiriKravitz,JainKravitz} that we use, and proves the reduction lemma. Section \ref{sec:class} classifies the critical subtori for five and six speeds. Section \ref{sec:fast} treats the one-very-fast-runner subtori and Section \ref{sec:UC} the subtorus $\UC$. Section \ref{sec:cons} draws the consequences for the spectra and reports the exhaustive searches. Section \ref{sec:code} describes the code.

\section{Preliminaries}\label{sec:prelim}

\subsection{Notation}
Throughout, $\dist{x}$ is the distance from the real number $x$ to the nearest integer, and $\T^n=(\R/\Z)^n$. For a closed subset $X\subseteq\T^n$ let
\[
\ML(X):=\sup_{x\in X}\min_{1\le r\le n}\dist{x_r},\qquad \Dd(X):=\tfrac12-\ML(X),
\]
so that $\Dd(X)$ is the $L^\infty$-distance from $X$ to the point $(1/2,\dots,1/2)$, as in \cite{GiriKravitz,JainKravitz}. A subtorus $T\subseteq\T^n$ is \emph{proper} if it is not contained in the union of the coordinate hyperplanes, equivalently if $\Dd(T)<1/2$. For integers $v_1,\dots,v_n$, not all zero, let $\langle v\rangle_\R$ denote the image in $\T^n$ of the line $\R v$; it is a one-dimensional subtorus, proper if and only if all $v_r$ are nonzero, and
\[
\ML(v_1,\dots,v_n)=\ML(\langle v\rangle_\R).
\]
Signs and the order of the speeds are irrelevant, and so is a common factor, so we may speak of the $\ML$ of a set of positive integers. A tuple with $\ML=1/(n+1)$ is \emph{tight}. A tuple with $\ML<1/n$ is \emph{near-tight}; for such a tuple we write $\ML=p/q$ in lowest terms and $k:=q-np\ge1$. For $n=6$ we shall find that $k$ divides $q$, and then we write $P:=q/k$. The following observation translates between the two descriptions.

\begin{lemma}\label{lem:Pk}
Let $P\ge2$ be an integer and $\ML=(P-1)/(6P)$. Then $k=6/\gcd(P-1,6)$. In particular, if $P\equiv1\pmod6$ then $k=1$ and $q=P$; if $P\equiv5\pmod6$ then $k=3$ and $q=3P$; and $P\equiv\pm1\pmod 6$ implies that $q$ is odd.
\end{lemma}
\begin{proof}
With $g=\gcd(P-1,6P)=\gcd(P-1,6)$ we have $p=(P-1)/g$, $q=6P/g$ and $k=q-6p=6/g$.
\end{proof}

\subsection{Subtori and their parameters}
Let $U=\langle u,v\rangle_\R\subseteq\T^n$ be a two-dimensional subtorus, the image of the plane spanned by $u,v\in\Z^n$. Following \cite[Section 2.2]{JainKravitz} we always choose $u,v$ such that
\begin{equation}\label{eq:saturated}
\langle u,v\rangle_\R\cap\Z^n=\langle u,v\rangle_\Z ;
\end{equation}
then $(s,t)\mapsto su+tv$ induces an isomorphism $\R^2/\Z^2\to U$, and the one-dimensional subtori of $U$ are exactly the $T=\langle Au+Bv\rangle_\R$ with $(A,B)$ coprime, two pairs giving the same subtorus if and only if they differ by a sign. We call $(A,B)$ the \emph{parameters} of $T$ and $w=Au+Bv$ its \emph{speed vector}; $T$ is proper with $n$ distinct speeds if and only if the entries of $w$ are nonzero and pairwise distinct in absolute value. The \emph{coordinate forms} of $U$ are the linear forms $\varphi_r(s,t)=su_r+tv_r$, identified with the integer vectors $(u_r,v_r)\in\Z^2$, $1\le r\le n$. A vector $su+tv\in U$ has vanishing $r$-th entry exactly when $(s,t)\in\ker\varphi_r$, and satisfies $|w_p|=|w_q|$ exactly when $(s,t)\in\ker(\varphi_p-\varepsilon\varphi_q)$ for some sign $\varepsilon$.

\subsection{Optimal times}
The function $t\mapsto\min_r\dist{tv_r}$ is continuous and piecewise linear, and its maximum is attained where two runners coincide or are antipodal. This is Proposition 2.1 of \cite{Kravitz}; we record the form in which we use it.

\begin{lemma}\label{lem:pairtimes}
Let $v_1,\dots,v_n$ be nonzero integers with $\ML(v)<1/2$. Then
\[
\ML(v)=\max\Bigl\{\min_r\dist{mv_r/L}: 1\le i<j\le n,\ \varepsilon\in\{+,-\},\ L=v_i-\varepsilon v_j\ne0,\ 0\le m<|L|\Bigr\}.
\]
\end{lemma}
\begin{proof}
Let $f(t)=\min_r\dist{tv_r}$ and let $t_0$ be a maximum of $f$. If only one runner $i$ attains the minimum at $t_0$, then $f$ agrees near $t_0$ with $\dist{tv_i}$, which has a local maximum only where $tv_i\equiv1/2$, giving $f(t_0)=1/2$, excluded. So two runners $i\ne j$ attain the minimum at $t_0$, that is $\dist{t_0v_i}=\dist{t_0v_j}$, which means $t_0v_i\equiv\varepsilon t_0v_j\pmod 1$ for some sign, so that $t_0(v_i-\varepsilon v_j)\in\Z$. If $v_i=\varepsilon v_j$ the two functions coincide and we may replace $j$ by another runner attaining the minimum; if all runners attaining the minimum have speeds equal up to sign to $v_i$ we are back in the first case. Hence $t_0=m/L$ for some pair with $L\ne0$.
\end{proof}

In the language of subtori, Lemma \ref{lem:pairtimes} says that $\Dd(T)=\min_{i,j,\varepsilon}\Dd(T\cap\{x_i=\varepsilon x_j\})$, which is the form used in \cite{JainKravitz}. It also yields an exact algorithm for $\ML$: scan $t=m/L$ over the candidate denominators $L\in\{v_i+v_j,|v_i-v_j|,2v_i\}$ and compare the rational numbers $\min_r\dist{mv_r/L}$ exactly. This is the algorithm implemented in our code.

\subsection{The reduction to critical subtori}
We use the following results of Giri and Kravitz. Let $S_k(n)$ be the set of values of $\Dd$ on the $k$-dimensional proper subtori of $\T^n$; a subtorus $T$ of $\T^n$ has a volume $\mathrm{vol}(T)$, equal to $\sqrt{v_1^2+\dots+v_n^2}$ when $T=\langle v\rangle_\R$ with $\gcd(v_1,\dots,v_n)=1$.

\begin{theorem}[{\cite[Theorem 1.4, Lemmas 3.3, 4.2 and 4.4]{GiriKravitz}}]\label{thm:GK}
\leavevmode
\begin{enumerate}
\item[(i)] For $1\le k<n$ the set $S_k(n)$ has no lower accumulation points: for every $d\in S_k(n)$ there is $\eta>0$ with $S_k(n)\cap(d-\eta,d)=\emptyset$.
\item[(ii)] For $1\le\ell<k\le n$, $\max S_k(n)=\max S_{k-\ell}(n-\ell)$.
\item[(iii)] For every $V>0$ there are only finitely many one-dimensional subtori of $\T^n$ of volume at most $V$.
\item[(iv)] For every $\varepsilon>0$ there is a constant $C^*(n,\varepsilon)$ such that every one-dimensional subtorus $T$ of volume greater than $C^*(n,\varepsilon)$ is contained in a subtorus $U$ of dimension at least $2$ in which it is $\varepsilon$-dense for the $L^2$-norm.
\end{enumerate}
\end{theorem}

\begin{lemma}[Reduction lemma]\label{lem:reduction}
Let $n\ge4$ and assume the Lonely Runner Conjecture for $n-1$ and for $n-2$ speeds. Then there is a finite set $E_n$ of primitive near-tight $n$-tuples such that for every near-tight $n$-tuple $v\notin E_n$ the subtorus $\langle v\rangle_\R$ is contained in a two-dimensional proper subtorus $U\subseteq\T^n$ with $\Dd(U)=1/2-1/n$.
\end{lemma}
\begin{proof}
By Theorem \ref{thm:GK}(ii) and the conjecture for $n-1$ speeds, $\max S_2(n)=\max S_1(n-1)=1/2-1/n$; by (ii) and the conjecture for $n-2$ speeds, $\max S_3(n)=1/2-1/(n-1)$. By (i) there is $\eta>0$ with $S_2(n)\cap(1/2-1/n-\eta,1/2-1/n)=\emptyset$. Fix $\varepsilon<\min(\eta,\,1/(n-1)-1/n)$ and let $E_n$ consist of the primitive near-tight tuples $v$ with $\mathrm{vol}\langle v\rangle_\R\le C^*(n,\varepsilon)$, a finite set by (iii). For $v\notin E_n$, (iv) provides a subtorus $U\supseteq T=\langle v\rangle_\R$ of dimension at least $2$ in which $T$ is $\varepsilon$-dense. Then $U$ is proper, since $T$ is, and $\Dd(U)\ge\Dd(T)-\varepsilon>1/2-1/n-\varepsilon$, because every point of $U$ is within $L^\infty$-distance $\varepsilon$ of a point of $T$. If $\dim U\ge3$ we would have $\Dd(U)\le1/2-1/(n-1)<1/2-1/n-\varepsilon$; so $\dim U=2$ and $\Dd(U)\in S_2(n)\cap(1/2-1/n-\varepsilon,1/2-1/n]=\{1/2-1/n\}$.
\end{proof}

The hypotheses hold for $n\le 10$ by the results quoted in the introduction. Nothing in the proof is effective as written: $\eta$ is not explicit.

\subsection{Tight instances}
We use the following classifications, all up to dilation: the only tight triple is $(1,2,3)$ (Cusick, see \cite{PerarnauSerra}); the only tight quadruples are $(1,2,3,4)$ and $(1,3,4,7)$ (Chen \cite{Chen91}; see also \cite{BarajasSerra09} and \cite[Section 4]{PerarnauSerra}); the only tight quintuples are $(1,2,3,4,5)$ and $(1,3,4,5,9)$ (Bohman, Holzman and Kleitman \cite[Theorem 3]{BHK}). We write $\mathcal T_4=\{(1,2,3,4),(1,3,4,7)\}$ and $\mathcal T_5=\{(1,2,3,4,5),(1,3,4,5,9)\}$.

\section{The critical subtori}\label{sec:class}

In this section $n\in\{5,6\}$, and $U\subseteq\T^n$ is a two-dimensional proper subtorus with $\Dd(U)=1/2-1/n$, that is, $\ML(U)=1/n$. We show that $U$ is spanned by two vectors of a very restricted shape and then enumerate. The arguments use only the Lonely Runner Conjecture for at most $n-1$ speeds and the classification of the tight $(n-1)$-tuples.

\begin{lemma}\label{lem:nolower}
There is no two-dimensional proper subtorus $W\subseteq\T^{n-1}$ with $\Dd(W)=1/2-1/n$.
\end{lemma}
\begin{proof}
Suppose $\ML(W)=1/n$. Every primitive $w\in W\cap\Z^{n-1}$ satisfies $\ML(w)\le\ML(W)=1/n$. If $w$ has no vanishing entry and its entries take at most $n-2$ distinct absolute values, then $\ML(w)\ge1/(n-1)>1/n$ by the Lonely Runner Conjecture for at most $n-2$ speeds, a contradiction; so the entries take exactly $n-1$ distinct absolute values, $\ML(w)\ge1/n$ by the conjecture for $n-1$ speeds, hence $\ML(w)=1/n$ and $w$ is a tight $(n-1)$-tuple up to signs, with entries at most $9$ in absolute value. Now $W\cap\Z^{n-1}$ is a lattice of rank two, so it contains infinitely many primitive directions, and at most $n-1$ of them lie on the lines $\ker\varphi_r$ where an entry vanishes. Infinitely many primitive vectors of bounded height in a rank-two lattice is impossible.
\end{proof}

\begin{corollary}\label{cor:distinct}
$U$ is contained in no subspace $\{x_p=\varepsilon x_q\}$ with $p\ne q$; equivalently, the coordinate forms $\varphi_1,\dots,\varphi_n$ of $U$ are pairwise distinct up to sign.
\end{corollary}
\begin{proof}
If $U\subseteq\{x_p=\varepsilon x_q\}$, let $\pi$ delete the $q$-th coordinate. Then $\pi$ is injective on $U$, so $W=\pi(U)$ is a two-dimensional subtorus of $\T^{n-1}$, and $\Dd(W)=\Dd(U)$ because $\dist{x_q}=\dist{x_p}$ on $U$; in particular $W$ is proper. This contradicts Lemma \ref{lem:nolower}.
\end{proof}

\begin{lemma}\label{lem:shape}
Let $w\in U\cap\Z^n$ be primitive, with no vanishing entry and with $|w_p|=|w_q|$ for some $p\ne q$. Then the absolute values of the entries of $w$ take exactly $n-1$ distinct values, which form a tight $(n-1)$-tuple in $\mathcal T_{n-1}$.
\end{lemma}
\begin{proof}
As in Lemma \ref{lem:nolower}: $\ML(w)\le1/n$, at most $n-2$ distinct values would give $\ML(w)\ge1/(n-1)$, and exactly $n-1$ distinct values give a tight instance.
\end{proof}

We call such a $w$ a \emph{tight repeat vector}. The line $\ker(\varphi_p-\varepsilon\varphi_q)$ in the parameter plane is a \emph{repeat line}; by Corollary \ref{cor:distinct} the form $\varphi_p-\varepsilon\varphi_q$ is nonzero, so a repeat line is a genuine line. A repeat line is \emph{free} if it is none of the $n$ vanishing lines $\ker\varphi_r$. The primitive vector of $U$ on a free repeat line is a tight repeat vector, and two non-parallel tight repeat vectors span $U$. The next three lemmas show that $U$ always has two free repeat lines.

\begin{lemma}\label{lem:count}
Let $d$ be the number of classes of the coordinate forms under proportionality, and let $m_1,\dots,m_d$ be the sizes of the classes. For any two classes $k\ne l$ the number of free repeat lines is at least $2\max(m_k,m_l)-(d-2)$. Consequently $U$ has at least two free repeat lines when $n=6$ and $d\le4$, or $n=5$ and $d\le4$; and at least one when $n=6$ and $d=5$.
\end{lemma}
\begin{proof}
Write $\varphi_r=c_r\psi_{\kappa(r)}$ with $\psi_1,\dots,\psi_d$ pairwise non-proportional and $c_r\ne0$. By Corollary \ref{cor:distinct} the numbers $|c_r|$ are pairwise distinct within each class. For $p$ in class $k$ and $q$ in class $l$, the repeat line $\ker(\varphi_p-\varepsilon\varphi_q)=\ker(\psi_k-\rho\psi_l)$ with $\rho=\varepsilon c_q/c_p$. The map $\rho\mapsto\ker(\psi_k-\rho\psi_l)$ is injective on $\Q^\times$ and never produces $\ker\psi_k$ or $\ker\psi_l$. For fixed $p$ the $2m_l$ values $\pm c_q/c_p$ are distinct, and for fixed $q$ the $2m_k$ values $\pm c_q/c_p$ are distinct, so at least $2\max(m_k,m_l)$ distinct repeat lines arise from the pair of classes, and at most $d-2$ of them are vanishing lines. Since $m_1+\dots+m_d=n$, some class has size at least $n/d$; for $n=6$ this gives at least $6,3,2,1$ free lines for $d=2,3,4,5$, and for $n=5$ at least $6,3,2$ for $d=2,3,4$.
\end{proof}

\begin{lemma}\label{lem:plucker}
Suppose the coordinate forms of $U$ are pairwise non-proportional ($d=n$). Then $U$ has a free repeat line.
\end{lemma}
\begin{proof}
Suppose not. Then for every $p<q$ and every sign $\varepsilon$ the form $\varphi_p-\varepsilon\varphi_q$ is proportional to some $\varphi_m$, and $m\notin\{p,q\}$ since $d=n$. Change coordinates in the parameter plane by an element of $\mathrm{GL}_2(\Q)$; this multiplies all determinants $\det(\varphi_i,\varphi_j)$ by the same nonzero constant and preserves all proportionalities, so we may assume $\varphi_1=(1,0)$ and $\varphi_2=(0,1)$. Describe a nonzero form $(\alpha,\beta)$ by the point $\xi=\beta/\alpha$ of the projective line, so that $\varphi_1,\varphi_2$ sit at $\xi=0,\infty$. The pair $\{1,2\}$ forces $(1,-1)$ and $(1,1)$ to be proportional to two further forms, which we label $\varphi_3=c(1,1)$ and $\varphi_4=d(1,-1)$ with $c,d\ne0$; they sit at $\xi=1,-1$. The pairs $\{1,3\}$ and $\{2,3\}$ force the four points
\[
t=\frac{c-1}{c},\quad \frac1t,\quad s=\frac{1+c}{c},\quad \frac1s
\]
to be occupied by coordinate forms other than $\varphi_1,\varphi_3$ (respectively $\varphi_2,\varphi_3$). One checks directly that each of $t,1/t,s,1/s$ lies in $\{0,\infty,1,-1\}$ only when $c\in\{1,\tfrac12,-1,-\tfrac12\}$, and that for $c$ outside this set the four points are distinct. Since only $n-4$ forms remain, and $n-4\le2$, we must have $c\in\{\pm1,\pm\tfrac12\}$, and then the pairs $\{1,3\},\{2,3\}$ require the two new points $\{2,\tfrac12\}$ (if $c=\pm1$) or $\{3,\tfrac13\}$ (if $c=\pm\tfrac12$). Symmetrically $d\in\{\pm1,\pm\tfrac12\}$ and the pairs $\{1,4\},\{2,4\}$ require $\{-2,-\tfrac12\}$ or $\{-3,-\tfrac13\}$. For $n=5$ two new points are required and one form remains; for $n=6$ four distinct new points are required and two forms remain. In both cases this is impossible.
\end{proof}

\begin{figure}[htbp]
\centering
\includegraphics[width=0.78\textwidth]{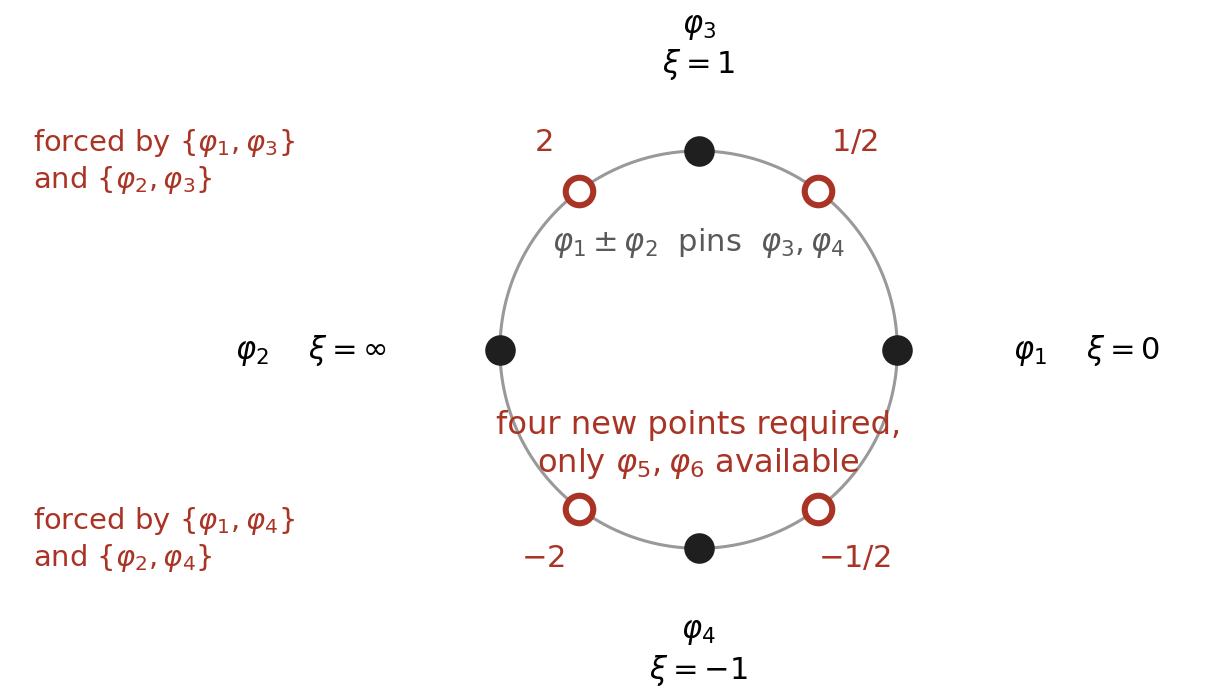}
\caption{The projective line of directions in the proof of Lemma \ref{lem:plucker}, drawn for $c=\pm1$, $d=\pm1$ and $n=6$: after $\varphi_1,\dots,\varphi_4$ are pinned at $0,\infty,1,-1$, the pairs $\{1,3\},\{2,3\}$ force the points $2,1/2$ and the pairs $\{1,4\},\{2,4\}$ the points $-2,-1/2$, while only two forms are left.}
\end{figure}

\begin{lemma}\label{lem:onetwo}
If $U$ has one free repeat line, it has two.
\end{lemma}
\begin{proof}
Let $w$ be the tight repeat vector on the given line. By Lemma \ref{lem:shape}, after a signed permutation of the coordinates, $w=(a,a,b_3,\dots,b_n)$ with $\{a,b_3,\dots,b_n\}\in\mathcal T_{n-1}$ and $b_3<\dots<b_n$; there are $10$ such $w$ for $n=6$ and $8$ for $n=5$. Let $z$ be any second generator, so that $\varphi_r=(w_r,z_r)$. Because the $|w_r|$ are distinct except for $|w_1|=|w_2|$, among the $2\binom n2$ triples $(p,q,\varepsilon)$ only $(1,2,+)$ has $w_p-\varepsilon w_q=0$, and its repeat line is the line of $w$. If no other repeat line were free, then for each of the remaining $2\binom n2-1$ triples the line $\ker(\varphi_p-\varepsilon\varphi_q)$ would coincide with some $\ker\varphi_r$, that is
\[
(w_p-\varepsilon w_q)z_r-w_r z_p+\varepsilon w_r z_q=0\qquad\text{for some } r .
\]
For a fixed triple this says that $z$ lies in the union of $n$ hyperplanes of $\Q^n$. We compute the intersection of these $2\binom n2-1$ unions exactly by branching on $r$, keeping every intermediate solution space in reduced row echelon form so that coinciding branches merge, and discarding a branch as soon as its solution space is contained in $\Q w$ (there $U$ would not be two-dimensional). The recursion terminates with the empty set for each of the $10$ (respectively $8$) vectors $w$. So a second free repeat line exists.
\end{proof}

The computation in Lemma \ref{lem:onetwo} is exact, over $\Q$, and takes a few seconds; the ten (eight) case lists are exhaustive by Lemma \ref{lem:shape}.

\begin{theorem}\label{thm:three}
Up to permutations and sign changes of the coordinates, the two-dimensional proper subtori $U\subseteq\T^6$ with $\Dd(U)=1/3$ are $\UA$, $\UB$ and $\UC$.
\end{theorem}
\begin{proof}
By Lemmas \ref{lem:count}, \ref{lem:plucker} and \ref{lem:onetwo}, $U$ contains two non-parallel tight repeat vectors, hence is spanned by them. A tight repeat vector is a quintuple of $\mathcal T_5$ with one value repeated in two of the six positions, with arbitrary signs; there are $115\,200$ of them up to a global sign, and up to signed permutations the first vector may be taken to be one of the ten canonical $w=(a,a,b_3,\dots,b_6)$. Of the $1\,151\,990$ non-parallel pairs, a rational-witness filter discards all but $190$: for each discarded pair $(w_1,w_2)$ the code exhibits a point $(p/N,q/N)$, $N=102$, at which all six coordinates of $pw_1+qw_2$ have $\dist{\cdot}>1/6$, which certifies $\ML(U)>1/6$. Each of the $190$ survivors is then decided exactly by a sweep: for fixed $s$ the set of $t$ with all $\dist{su_r+tv_r}>1/6$ is an intersection of finite unions of open intervals with rational endpoints, whose combinatorics changes only at finitely many rational $s$, so testing one $s$ in each gap decides whether $\ML(U)>1/6$. All $190$ have $\ML(U)=1/6$ (the inequality $\ML(U)\ge1/6$ holds because $U$ contains a tight quintuple). They span $22$ distinct planes, which fall into $3$ orbits under the $46\,080$ signed permutations, represented by $\UA$, $\UB$ and $\UC$; the representatives are identified by computing a canonical form of each orbit. Conversely $\Dd(\UA)=\Dd(\UB)=\Dd(\UC)=1/3$ by the same decision procedure.
\end{proof}

The generators given for $\UA,\UB,\UC$ satisfy \eqref{eq:saturated}; for $\UC$ this is the basis used in \cite[Section 6]{JainKravitz}, where the subtorus is called $U^7$.

\begin{theorem}\label{thm:two}
Up to permutations and sign changes of the coordinates, the two-dimensional proper subtori $U\subseteq\T^5$ with $\Dd(U)=3/10$ are $\langle(1,2,3,4,0),e_5\rangle_\R$ and $\langle(1,3,4,7,0),e_5\rangle_\R$.
\end{theorem}
\begin{proof}
Identical, with $\mathcal T_4$ in place of $\mathcal T_5$: there are $7\,680$ tight repeat vectors up to sign and $8$ canonical first vectors; the witness filter (with $N=60$) leaves $116$ pairs, of which $112$ have $\ML(U)=1/5$; they span $16$ planes forming $2$ orbits under the $3\,840$ signed permutations.
\end{proof}

\section{The one-very-fast-runner subtori}\label{sec:fast}

Let $n\in\{5,6\}$, $V\in\mathcal T_{n-1}$, and $U_V=\langle(V,0),e_n\rangle_\R$. The generators satisfy \eqref{eq:saturated}, so the proper one-dimensional subtori of $U_V$ with $n$ distinct speeds are exactly the tuples $bV\cup\{c\}$ with $b,c\ge1$ coprime and $c\notin bV$. We need two facts about the four tight instances in $\mathcal T_4\cup\mathcal T_5$, both immediate from the definition: none of their entries is divisible by $n$, and their equality times $\{t\in[0,1):\min_i\dist{tv_i}=1/n\}$ are exactly the fractions $m/n$ with $\gcd(m,n)=1$, at each of which the runner of speed $1$ is at position $1/n$ or $1-1/n$.

\begin{proposition}[pre-jump]\label{prop:prejump}
Let $V\in\mathcal T_{n-1}$, $b\ge2$, and let $c$ be coprime to $b$ with $c\notin bV$. Then $\ML(bV\cup\{c\})\ge1/n$.
\end{proposition}
\begin{proof}
At $t=1/(nb)$ the runners of $bV$ are at the positions $v_i/n$, none of which is an integer, so each is at distance at least $1/n$ from $0$; they do not move under $t\mapsto t+h/b$, $h\in\Z$. The remaining runner is at $c/(nb)+ch/b$, and as $h$ runs over $0,\dots,b-1$ these positions form a coset of the cyclic group $\frac1b\Z/\Z$, since $\gcd(c,b)=1$. Some element of a coset of that group lies within $1/(2b)$ of $1/2$, hence at distance at least $1/2-1/(2b)\ge1/4\ge1/n$ from $0$.
\end{proof}

\begin{theorem}\label{thm:fast}
Let $n\in\{5,6\}$, $V\in\mathcal T_{n-1}$ and $c\ge1$, $c\notin V$. Then $\ML(V\cup\{c\})<1/n$ if and only if $n\mid c$, and in that case $\ML(V\cup\{c\})=s/(ns+1)$ with $s=c/n$, except for $(1,3,4,5,9,6)$ and $(1,3,4,7,5)$, whose values are $2/13$ and $2/11$. In particular every near-tight one-dimensional subtorus of $U_V$ has $k=1$.
\end{theorem}
\begin{proof}
If $n\nmid c$ then at $t=1/n$ all runners are at distance at least $1/n$ from $0$, so $\ML\ge1/n$. Let $c=ns$. Lemma 9.4 of \cite{Kravitz} applies to $V$ with $L=1/n$: the least common multiple of the denominators of the equality times is $n$, and the quantity $v_\mu$ of that lemma is the speed $1$, since at every equality time the runner of speed $1$ occupies one of the positions $\pm1/n$. The lemma gives $\ML(V\cup\{c\})=cL/(c+v_\mu)=s/(ns+1)$ for every multiple $c$ of $n$ exceeding $4v_{n-1}^3v_1=4\max(V)^3$, that is $c>500$, $2916$, $256$, $1372$ for $V=(1,2,3,4,5)$, $(1,3,4,5,9)$, $(1,2,3,4)$, $(1,3,4,7)$ respectively. For $c\le30\,000$ (when $n=6$) and $c\le20\,000$ (when $n=5$) the value was computed exactly with the algorithm of Lemma \ref{lem:pairtimes}: it is $s/(ns+1)$ in every case except the two listed, where $k$ still equals $1$. Together with Proposition \ref{prop:prejump} this covers every proper one-dimensional subtorus of $U_V$ with $n$ distinct speeds.
\end{proof}

\section{The relative spectrum of \texorpdfstring{$\UC$}{U\_C}}\label{sec:UC}

In this section $U=\UC=\langle u,v\rangle_\R$ with $u=(1,0,1,2,3,3)$, $v=(0,1,1,1,1,2)$, and $T=\langle Au+Bv\rangle_\R$ with $(A,B)$ coprime, so that the speed vector is
\[
w=Au+Bv=(A,\ B,\ A+B,\ 2A+B,\ 3A+B,\ 3A+2B).
\]
We assume throughout that the six entries are nonzero and distinct in absolute value, and we want to compute $\ML(w)$ exactly for all such $(A,B)$. We follow \cite[Sections 2.2--2.3]{JainKravitz}, adding explicit thresholds; the reader may find it useful to compare with their Figures 22 and 23, which tabulate the same objects for this subtorus.

\subsection{Components and cosets}
Fix $1\le i<j\le6$ and a sign $\varepsilon$, and put
\[
(a,b):=(u_i-\varepsilon u_j,\ v_i-\varepsilon v_j)\ne(0,0),\qquad K:=\gcd(a,b),\qquad (a',b'):=(a,b)/K,
\]
and choose $x_0,y_0\in\Z$ with $a'x_0+b'y_0=1$. In the parameter plane $\R^2/\Z^2\cong U$ the subgroup $U_{ij\varepsilon}:=U\cap\{x_i=\varepsilon x_j\}$ is $\{(s,t): as+bt\in\Z\}$, the union of the $K$ circles
\[
C_\ell=\bigl\{\tfrac{\ell}{K}(x_0,y_0)+\theta(-b',a'):\ \theta\in\R/\Z\bigr\},\qquad 0\le\ell<K .
\]
On $C_\ell$ the $r$-th coordinate is $x_r=c_r+\theta d_r$ with $c_r=\frac\ell K(x_0u_r+y_0v_r)$ and $d_r=a'v_r-b'u_r$, so the restriction of $\min_r\dist{x_r}$ to $C_\ell$ is the piecewise linear function
\[
f_\ell(\theta)=\min_{1\le r\le6}\dist{c_r+\theta d_r},\qquad \theta\in\R/\Z .
\]
Its maximum is at most $\ML(U)=1/6$. Call the quadruple $(i,j,\varepsilon,\ell)$ \emph{critical} if the maximum equals $1/6$, and let $\mathcal Y$ be the set of critical quadruples. For the $15\cdot2=30$ pairs $(i,j,\varepsilon)$ there are $42$ components in all, $18$ of them critical; on each of the other $24$ the function $f_\ell$ vanishes identically. Table \ref{tab:Y} lists the critical quadruples; it agrees with Figure 22 of \cite{JainKravitz}.

Now let $L:=aA+bB=w_i-\varepsilon w_j$. By assumption $L\ne0$. The points of $T$ on $U_{ij\varepsilon}$ are the $\lambda w$ with $\lambda L\in\Z$, i.e.\ $\lambda=m/L$, and the point $\lambda w$ lies on $C_\ell$ exactly when $m\equiv\ell\pmod K$. Put
\[
q:=|L|/K,\qquad c:=x_0B-y_0A,\qquad \theta_0:=\ell c/L\bmod1 .
\]
The unimodular matrix $\begin{psmallmatrix}a'&b'\\-y_0&x_0\end{psmallmatrix}$ sends $(A,B)$ to $(L/K,c)$, so $\gcd(q,c)=1$, and in the coordinate $\theta$ on $C_\ell$ the points of $T\cap C_\ell$ form the coset $\theta_0+\frac1q\Z/\Z$. By Lemma \ref{lem:pairtimes},
\begin{equation}\label{eq:ML-cosets}
\ML(w)=\max_{(i,j,\varepsilon,\ell)}\ \max_{0\le m'<q}\ f_\ell\bigl(\theta_0+m'/q\bigr),
\end{equation}
and only the critical quadruples can contribute a value above $0$. We write $q_Y$, $\theta_Y$, $L_Y$ for the quantities attached to $Y\in\mathcal Y$; note that $L_Y$ is a linear form in $(A,B)$ (column two of Table \ref{tab:Y}) and $q_Y=|L_Y|/K_Y$.

\begin{table}[htbp]
\centering\small
\begin{tabular}{@{}llccll@{}}
\toprule
$(i,j,\varepsilon,\ell)$ & $L=w_i-\varepsilon w_j$ & $K$ & maximisers $\tau_h$ of $f_\ell$ & $q_0$\\
\midrule
$(1,2,+,0)$ & $A-B$ & 1 & $1/6,\,5/6$ & 43\\
$(1,5,-,0)$ & $4A+B$ & 1 & $1/6,\,5/6$ & 43\\
$(1,6,+,1)$ & $-2A-2B$ & 2 & $1/6,\,1/3,\,2/3,\,5/6$ & 12\\
$(1,6,-,1)$ & $4A+2B$ & 2 & $1/6,\,1/3,\,2/3,\,5/6$ & 12\\
$(2,3,-,0)$ & $A+2B$ & 1 & $1/6,\,5/6$ & 43\\
$(2,4,+,1)$ & $-2A$ & 2 & $1/6,\,1/3,\,2/3,\,5/6$ & 12\\
$(2,4,-,1)$ & $2A+2B$ & 2 & $1/6,\,1/3,\,2/3,\,5/6$ & 12\\
$(2,5,+,1)$, $(2,5,+,2)$ & $-3A$ & 3 & $1/6,\,5/6$ & 19\\
$(2,6,-,1)$, $(2,6,-,2)$ & $3A+3B$ & 3 & $1/2,\,5/6$ and $1/6,\,1/2$ & 19\\
$(3,5,+,1)$ & $-2A$ & 2 & $1/6,\,1/3,\,2/3,\,5/6$ & 12\\
$(3,5,-,1)$ & $4A+2B$ & 2 & $1/6,\,1/3,\,2/3,\,5/6$ & 12\\
$(3,6,-,0)$ & $4A+3B$ & 1 & $1/6,\,5/6$ & 43\\
$(4,5,-,0)$ & $5A+2B$ & 1 & $1/6,\,5/6$ & 43\\
$(4,6,-,0)$ & $5A+3B$ & 1 & $1/6,\,5/6$ & 43\\
$(5,6,-,1)$, $(5,6,-,2)$ & $6A+3B$ & 3 & $1/6,\,1/2$ and $1/2,\,5/6$ & 19\\
\bottomrule
\end{tabular}
\caption{The $18$ critical quadruples of $\UC$, the linear forms $L$, the number $K$ of components, the maximisers of $f_\ell$ on the component, and the threshold $q_0$ of Lemma \ref{lem:coset}. The nine distinct directions of the forms $L$ are the rays of slope $B/A\in\{-4,-5/2,-2,-5/3,-4/3,-1,-1/2,1\}$ and the line $A=0$.}
\label{tab:Y}
\end{table}

\subsection{The maximum over a coset}
Fix a critical quadruple and drop the index $\ell$. Let $\tau_1,\dots,\tau_H$ be the points where $f$ attains its maximum $1/6$ (they are rational, and there is no interval on which $f$ is constant $1/6$), and for each $h$ let $\lambda_h^{\pm}>0$ and $\rho_h^{\pm}>0$ be such that
\[
f(\tau_h+t)=\tfrac16-\lambda_h^+t\ \ (0\le t\le\rho_h^+),\qquad f(\tau_h-t)=\tfrac16-\lambda_h^-t\ \ (0\le t\le\rho_h^-),
\]
with $\rho_h^\pm$ maximal; these \emph{zones} are read off from the piecewise linear description of $f$. Let $\rho_{\min}=\min_h\min(\rho_h^-,\rho_h^+)$, let $\lambda_s$ be the smallest of the slopes $\lambda_h^\pm$, and let $f_{\mathrm{out}}$ be the maximum of $f$ on the complement of the union of the open zones $(\tau_h-\rho_h^-,\tau_h+\rho_h^+)$; then $f_{\mathrm{out}}<1/6$. Put
\[
q_0:=\max\Bigl(\Bigl\lceil\frac1{\rho_{\min}}\Bigr\rceil,\ \Bigl\lfloor\frac{\lambda_s}{1/6-f_{\mathrm{out}}}\Bigr\rfloor+1\Bigr).
\]
The values for the $18$ critical quadruples are in Table \ref{tab:Y}: $q_0=43$ when $K=1$ (there $\rho_{\min}=1/30$, $\lambda_s=1$, $f_{\mathrm{out}}=1/7$), $q_0=12$ when $K=2$ and $q_0=19$ when $K=3$. For a rational $\tau$ and a coset $\theta_0+\frac1q\Z$ write, as in \cite[Section 2.3]{JainKravitz},
\[
\mathrm{Approx}^-(\tau,\theta_0;q)=(\tau-\theta_0)\bmod\tfrac1q,\qquad \mathrm{Approx}^+(\tau,\theta_0;q)=(\theta_0-\tau)\bmod\tfrac1q
\]
for the distances from $\tau$ to the nearest points of the coset on either side.

\begin{lemma}\label{lem:coset}
If $q\ge q_0$, then
\[
\max_{0\le m'<q}f(\theta_0+m'/q)=\frac16-\frac{\gamma}{q},\qquad
\gamma:=q\cdot\min_{h}\min\bigl(\lambda_h^-\mathrm{Approx}^-(\tau_h,\theta_0;q),\ \lambda_h^+\mathrm{Approx}^+(\tau_h,\theta_0;q)\bigr).
\]
\end{lemma}
\begin{proof}
Since $1/q\le\rho_{\min}$, the nearest coset points to the left and to the right of each $\tau_h$ lie inside the corresponding zone, at distances $\mathrm{Approx}^\pm<1/q$, and $f$ decreases linearly away from $\tau_h$ inside the zone; so the maximum of $f$ over the coset points lying in some zone is exactly $1/6-\gamma/q$. It remains to see that the overall maximum is attained in a zone. Take the maximiser $\tau_h$ and the side with slope $\lambda_s$; the nearest coset point on that side has $f\ge1/6-\lambda_s/q>f_{\mathrm{out}}$ by the choice of $q_0$, while every coset point outside the zones has $f\le f_{\mathrm{out}}$.
\end{proof}

This is Lemma 2.5 of \cite{JainKravitz} with an explicit threshold. The formula was checked against a direct evaluation of $f$ on the coset for $20\,000$ random pairs $(A,B)$ and all critical quadruples with $q\ge q_0$.

\begin{lemma}\label{lem:period}
For each critical quadruple, the number $\gamma$ of Lemma \ref{lem:coset} depends only on the residue class of $(A,B)$ modulo $6$ and on the sign of $L$.
\end{lemma}
\begin{proof}
With $\sigma=\mathrm{sign}(L)$ we have $q\theta_0=\ell c/(K\sigma)$ and $q\,\mathrm{Approx}^\mp(\tau_h,\theta_0;q)=\pm(q\tau_h-q\theta_0)\bmod1$. Every maximiser has the form $\tau_h=w_h/x_h$ with $x_h\in\{2,3,6\}$ (Table \ref{tab:Y}), so $q\tau_h\bmod1$ depends only on $q\bmod x_h$, and $\ell c/K\bmod1$ depends only on $c\bmod K$ with $K\in\{1,2,3\}$. Finally $q=\sigma(a'A+b'B)$ and $c=x_0B-y_0A$ are linear in $(A,B)$ with integer coefficients, so both residues are determined by $(A,B)$ modulo $6$ and by $\sigma$.
\end{proof}

We write $\gamma_Y(\bar A,\bar B,\sigma)$ for the resulting function on $(\Z/6)^2\times\{\pm1\}$; it is computed once and for all by evaluating the expression of Lemma \ref{lem:coset} at one representative of each class. (The modulus $36$ appearing in Proposition 2.6 of \cite{JainKravitz} is an upper bound; for $\UC$ the true period is $6$.)

\subsection{The far region}
Call $(A,B)$ \emph{far} if $q_Y\ge q_0(Y)$ for every critical quadruple $Y\in\mathcal Y$. For far parameters \eqref{eq:ML-cosets} and Lemma \ref{lem:coset} give
\begin{equation}\label{eq:far}
\ML(w)=\frac16-\min_{Y\in\mathcal Y}\frac{\gamma_Y}{q_Y},
\end{equation}
and if some $\gamma_Y$ vanishes then $\ML(w)=1/6$ and $T$ is not near-tight. The nine lines $L_Y=0$ cut the open half plane $A>0$ into nine open cones (the parameters may be normalised by $A\ge0$, and $A=0$ forces $(A,B)=(0,1)$, which has a vanishing speed). On a fixed residue class of $(A,B)$ modulo $6$ and a fixed cone, the signs $\sigma_Y$ and the numbers $\gamma_Y$ are constant, so the quantities
\[
F_Y(A,B):=\frac{q_Y}{\gamma_Y}=\frac{\sigma_YL_Y(A,B)}{K_Y\gamma_Y}
\]
are homogeneous linear forms with rational coefficients, and by \eqref{eq:far}
\begin{equation}\label{eq:P}
\ML(w)=\frac16-\frac1{6P},\qquad P:=\frac16\max_{Y}F_Y(A,B).
\end{equation}
Two linear forms compare in the same way throughout an open cone unless their difference changes sign, which happens along a ray through the origin. Cutting each cone along the finitely many rays where two of the forms $F_Y$ agree therefore produces open subcones on each of which a single $F_Y$ is maximal, and on each subcone (and on each cutting ray) $P$ is a fixed linear form. This is the content of Proposition 2.6 of \cite{JainKravitz}; in their language the subcones are the sectors, and the winner is the quadruple whose offset is smallest.

For the classes of $(A,B)$ modulo $6$ that contain coprime pairs there are $24\cdot9=216$ pairs (class, cone). On $108$ of them some $\gamma_Y$ vanishes, so no far parameter in them is near-tight. The other $108$ split into $852$ subcones and cutting rays in all. On each of these pieces the winning form $P=\alpha A+\beta B$ has $6\alpha,6\beta\in\Z$, so $P\bmod6$ is constant on each residue class of $(A,B)$ modulo $36$; refining accordingly gives $18\,216$ pieces, and on every one of them
\[
P\in\Z\qquad\text{and}\qquad P\equiv\pm1\pmod6 .
\]
The winning forms are few: $P$ is one of $A-B$, $A+2B$, $4A+B$, $4A+3B$, $5A+2B$, $5A+3B$, $-A-2B$, $-4A-3B$, $-5A-2B$, $-5A-3B$ (on pieces with $P\equiv1$), $(A-B)/4$ and $(A-B)/5$ (on two cutting rays), and $A$, $2A+B$, $-A-B$ (on pieces with $P\equiv5$). Reading off on which pieces $P\equiv5\pmod6$, and using Lemma \ref{lem:Pk}, gives exactly the table of Theorem B.

\subsection{The strips}
A parameter pair that is not far has $q_Y<q_0(Y)$ for some $Y$, hence lies on one of finitely many lattice lines $L_G=\ell_0$ parallel to a ray, where $L_G$ is the primitive form in the direction of $L_Y$ and $|\ell_0|$ is bounded by $K_Yq_0(Y)$; there are $510$ such lines. Fix one, parametrise it as $(A,B)=P_0+t\,r$ with $r$ the primitive direction vector, and let $G$ be the set of critical quadruples whose form is proportional to $L_G$. For $Y\in G$ the number $q_Y$ is constant along the line and small, so the maximum of $f_Y$ over the coset is a function of $t$ that can be computed directly from the definition, and it is periodic in $t$ with period dividing $|L_Y|$, because $\theta_Y$ depends only on $c\bmod|L_Y|$ and $c$ is linear in $t$. For $Y\notin G$ the form $L_Y(P_0+tr)$ is a non-constant linear function of $t$, so there is an explicit $T_0$ beyond which $q_Y\ge q_0(Y)$ and the sign of $L_Y$ is constant in each direction. For $|t|>T_0$ the value $\ML(w)$ is therefore known exactly from \eqref{eq:ML-cosets}: the quadruples of $G$ contribute an explicitly known periodic function of $t$, and the others contribute $\gamma_Y/q_Y(t)$ with $q_Y(t)$ linear. On each residue class of $t$ modulo a suitable period, comparing finitely many linear functions of $t$ decomposes the half-lines $t>T_0$ and $t<-T_0$ into finitely many intervals on which the winner is fixed, and on each interval $P$ is a linear function of $t$ with $P\in\Z$ and $P\bmod6$ constant on the class. This yields $331\,140$ such intervals; on none of them does a quadruple of $G$ win, and on all of them $P\equiv\pm1\pmod6$, with $k=3$ exactly where the table of Theorem B says so. The parameter pairs with $|t|\le T_0$ on the $510$ lines are $23\,214$ in number, all with $|A|,|B|\le299$; their $\ML$ was computed directly by the algorithm of Lemma \ref{lem:pairtimes}. Among them $7\,120$ are near-tight, and every one satisfies the conclusion of Theorem B.

\subsection{The theorem}
Every pair $(A,B)$ with six nonzero, distinct speeds is far or lies on one of the $510$ lines, so the two computations above cover every proper one-dimensional subtorus of $\UC$ with six distinct speeds.

\begin{theorem}\label{thm:UC}
Let $T\subseteq\UC$ be a one-dimensional subtorus with parameters $(A,B)$, $A>0$, whose six speeds are nonzero and distinct in absolute value, and suppose $\ML(T)<1/6$. Then $\ML(T)=(P-1)/(6P)$ with $P\in\Z$, $P\equiv\pm1\pmod6$; hence $k\in\{1,3\}$ and the denominator of $\ML(T)$ is odd. One has $k=3$ exactly in the four cases of the table in Theorem B. Moreover
\[
\max_i|w_i|<3q,
\]
and $3$ is the best constant: along $(A,B)=(18j+31,6)$, $j\ge0$, one has $\ML(T)=(3j+7)/(18j+43)$, $q=18j+43$ and $\max_i|w_i|=54j+105=3q-24$.
\end{theorem}
\begin{proof}
The first two assertions are the outcome of the computation described in the preceding two subsections, which is a proof because the only ingredients are Lemma \ref{lem:pairtimes}, Lemma \ref{lem:coset} with its explicit threshold, Lemma \ref{lem:period}, the exact decomposition of the cones and lines, and the exact evaluation of $\ML$ at the finitely many remaining parameters. For the inequality, on each far piece $\max_i|w_i|/q$ is the quotient of a maximum of linear forms by the linear form $q\in\{P,3P\}$, a monotone function of the slope $B/A$ on every interval where no speed changes sign; its supremum over the piece is therefore attained at an end of the piece or at a zero of a speed, and evaluating there gives $3$ as the largest value, reached only at the rays $B=0$, $3A+B=0$, $3A+2B=0$, where a speed vanishes and which contain no admissible parameters. On the $331\,140$ line intervals the same argument in the variable $t$ gives the supremum $63/23$, and the largest value among the $23\,214$ directly checked pairs is $141/55$. The family $(18j+31,6)$ lies in the class $(1,0)$ modulo $6$, on a piece where $P=A+2B$, so $q=P=18j+43$, and $\max_i|w_i|=3A+2B=3P-24$.
\end{proof}

\begin{remark}\label{rem:denominators}
On every near-tight subtorus of $\UC$ with $|A|,|B|\le300$ the denominator $q$ of $\ML$ coincides with the smallest denominator of a time at which the maximum is attained. This is not a general feature of near-tight tuples: the sporadic sextuple $(1,3,4,5,18,46)$ of Table \ref{tab:sporadic6} has $\ML=4/25$, attained only at the times $23/50$ and $27/50$.
\end{remark}

\begin{remark}
As a consistency check, formula \eqref{eq:far} was compared with the direct computation of $\ML$ for all $64\,146$ far parameter pairs with $|A|,|B|\le300$, with no discrepancy; and all $54\,857$ near-tight subtori of $\UC$ with $|A|,|B|\le300$ were checked directly against the statements of Theorem \ref{thm:UC}. Neither check is needed for the proof.
\end{remark}

\section{Consequences for the spectra}\label{sec:cons}

\subsection{Six speeds}
\begin{theorem}\label{thm:six}
There is a finite set $E_6$ of primitive near-tight sextuples with the following property: every near-tight sextuple $v\notin E_6$ satisfies $\ML(v)=(P-1)/(6P)$ for an integer $P\equiv\pm1\pmod6$; in particular $k\in\{1,3\}$ and the denominator of $\ML(v)$ is odd. No sextuple with all speeds at most $110$ is a counterexample to this conclusion.
\end{theorem}
\begin{proof}
Let $E_6$ be the set of Lemma \ref{lem:reduction} for $n=6$. For $v\notin E_6$ the subtorus $\langle v\rangle_\R$ lies in a two-dimensional $U$ with $\Dd(U)=1/3$, which by Theorem \ref{thm:three} is $\UA$, $\UB$ or $\UC$ after a signed permutation of the coordinates, an operation that changes neither $\ML$ nor the distinctness of the speeds. Theorems \ref{thm:fast} and \ref{thm:UC} apply. The last sentence is the exhaustive search described below.
\end{proof}

\begin{conjecture}\label{conj:parity}
$E_6$ contains no exception: for every near-tight sextuple the denominator of $\ML$ is odd, equivalently $k\in\{1,3\}$.
\end{conjecture}

\begin{remark}
The congruence $P\equiv\pm1\pmod6$ says more than the parity of $q$. The values $(P-1)/(6P)$ with $P\equiv3\pmod 6$, such as $4/27$, $7/45$ and $10/63$, have $k=3$ and are therefore permitted by the amended conjecture of Fan and Sun, yet Theorem \ref{thm:six} excludes them for all but finitely many sextuples and the search finds none below $110$. Both halves of the congruence, the exclusion of even $P$ and the exclusion of $P\equiv3$, come out of the same decomposition of Section \ref{sec:UC}: on every piece the winning form happens to take values in the classes $1$ and $5$ modulo $6$. The computation certifies this but does not explain it, and a conceptual reason why the classes $2$, $3$ and $4$ never win would be of interest.
\end{remark}

We now describe the searches. All $\binom{110}6=2\,141\,851\,635$ sextuples of distinct speeds at most $110$ were tested for near-tightness with the exact criterion of Lemma \ref{lem:pairtimes}, and $\ML$ was computed exactly for the near-tight ones. There are $348$ primitive near-tight sextuples; $315$ have $k=1$ and $33$ have $k=3$; for all of them $P$ is an integer congruent to $\pm1$ modulo $6$. Testing membership in the three critical subtori up to signed permutations, $18$ lie on $\UA$, $18$ on $\UB$, $304$ on $\UC$ (the sextuple $(1,3,4,5,6,9)$, with $\ML=2/13$, lies on both $\UB$ and $\UC$), and the nine of Table \ref{tab:sporadic6} on none of them. All $33$ sextuples with $k=3$ lie on $\UC$ and fall in the four classes of Theorem B; the smallest are $(1,5,6,11,16,17)$ with $\ML=5/33$ and $(5,6,11,17,23,28)$ with $\ML=8/51$, the latter being the example of Fan and Sun. The largest ratio $\max_i v_i/q$ among the $348$ is $105/43$, on the extremal family of Theorem \ref{thm:UC}.

\begin{table}[htbp]
\centering\small
\begin{tabular}{@{}lccc@{}}
\toprule
sextuple & $\ML$ & $k$ & $P$\\
\midrule
$(1,2,5,6,7,8)$ & $2/13$ & 1 & 13\\
$(2,5,6,8,10,11)$ & $2/13$ & 1 & 13\\
$(1,4,5,6,7,22)$ & $2/13$ & 1 & 13\\
$(1,2,3,5,8,18)$ & $3/19$ & 1 & 19\\
$(2,3,5,6,8,11)$ & $3/19$ & 1 & 19\\
$(2,5,6,7,8,11)$ & $3/19$ & 1 & 19\\
$(1,3,4,5,18,46)$ & $4/25$ & 1 & 25\\
$(1,3,4,5,7,24)$ & $5/31$ & 1 & 31\\
$(1,4,5,6,7,33)$ & $6/37$ & 1 & 37\\
\bottomrule
\end{tabular}
\caption{The near-tight primitive sextuples with speeds at most $110$ that lie on none of $\UA,\UB,\UC$. They belong to the exceptional set $E_6$ of Theorem \ref{thm:six}. Their values are all attained on the critical subtori as well.}
\label{tab:sporadic6}
\end{table}

Since the values in Table \ref{tab:sporadic6} are also attained on the critical subtori, the set of \emph{values} of $\ML$ on near-tight sextuples with speeds at most $110$ is contained in $\{(P-1)/(6P):P\equiv\pm1\pmod6\}$, and Theorem \ref{thm:six} shows that at most finitely many values can ever fall outside this set. We do not know whether the list of sporadic sextuples is finite as a list of tuples with $k=1$ only, or whether Table \ref{tab:sporadic6} is complete; the nine tuples are of two visible kinds, a near-tight quintuple with a further runner of moderate speed (for instance $(1,4,5,6,7)$, whose $\ML$ is $2/11$, with $22$ or $33$ adjoined) and tuples with no apparent structure.

\subsection{Five speeds}
\begin{theorem}\label{thm:five}
There is a finite set $E_5$ of primitive near-tight quintuples such that every near-tight quintuple $v\notin E_5$ satisfies $\ML(v)=s/(5s+1)$ for some $s\in\mathbb N$. No quintuple with all speeds at most $130$ is a counterexample to this conclusion.
\end{theorem}
\begin{proof}
Lemma \ref{lem:reduction} with $n=5$, Theorem \ref{thm:two} and Theorem \ref{thm:fast}.
\end{proof}

In the notation of \cite{JainKravitz}, $S_1(5)\cap(3/10,1/2]$ has finite symmetric difference with $3/10+\tfrac15\mathrm{Prog}(5,6)$. Among the $\binom{130}5$ quintuples of distinct speeds at most $130$ there are $64$ primitive near-tight ones, all with $k=1$; $52$ lie on the two critical subtori and the following twelve do not:
\[
\begin{array}{llll}
(1,2,3,5,8)\ [2/11], & (1,4,5,6,7)\ [2/11], & (1,3,4,5,18)\ [2/11], & (1,3,4,5,9)\ [1/6],\\
(1,3,4,5,11)\ [3/16], & (2,5,7,8,9)\ [3/16], & (1,7,8,9,15)\ [3/16], & (2,3,5,8,19)\ [4/21],\\
(3,5,8,9,13)\ [4/21], & (3,5,8,11,13)\ [4/21], & (1,3,4,5,27)\ [6/31], & (1,4,5,6,27)\ [6/31].
\end{array}
\]
(The tight quintuple $(1,3,4,5,9)$ is near-tight in our sense, since $1/6<1/5$.)

\subsection{Four, seven and eight speeds}
For four speeds the two critical subtori are, by \cite[Section 3.3]{JainKravitz},
\[
U^1=\langle(0,1,2,3),(1,0,0,0)\rangle_\R\qquad\text{and}\qquad U^2=\langle(1,0,1,1),(1,1,0,2)\rangle_\R,
\]
and the relative spectra are computed there: $k=1$ on $U^1$ and both $k=1$ and $k=2$ on $U^2$, the latter giving the family of Fan and Sun. Our search over the $\binom{200}4$ quadruples with speeds at most $200$ finds $4\,125$ primitive near-tight ones, $3\,836$ with $k=1$ and $289$ with $k=2$; $50$ lie on $U^1$, $4\,075$ on $U^2$, and exactly one, $(1,3,4,14)$ with $\ML=4/17$, on neither. Its value $\Dd=9/34=1/4+1/68$ lies in the progression $1/4+\tfrac14\mathrm{Prog}(2,3)$ of \cite[Theorem 1.3]{JainKravitz}, in accordance with the remark there that no exceptional values were found; the tuple, however, is exceptional in the sense of Lemma \ref{lem:reduction}.

For seven speeds the method is blocked by the missing classification of the tight sextuples; the tight instances with six speeds found by an exhaustive search up to speed $130$ are the dilations of $(1,2,3,4,5,6)$ only. Among the $\binom{50}7$ septuples with speeds at most $50$ there are $29$ primitive near-tight ones, $26$ with $k=1$ and three with $k=2$, namely $(1,2,3,4,5,7,18)$ and $(1,3,4,5,7,13,18)$ with $\ML=3/23$, and $(1,3,4,5,7,11,30)$ with $\ML=5/37$; the first and third are the examples of \cite[Remark 30]{FanSun}. Among the $\binom{40}8$ octuples with speeds at most $40$ the $29$ primitive near-tight ones all have $k=1$. So the parity phenomenon of Theorem \ref{thm:six} is specific to six speeds: even values of $k$ occur for four and for seven speeds.

\subsection{Questions}
\begin{question}
Is there a proof of Theorem \ref{thm:six} that avoids the classification of the critical subtori, for instance by tracing a parity constraint through the proof of Theorem 1.1 of \cite{JainKravitz}?
\end{question}

\begin{question}\label{q:ratio}
For a near-tight $n$-tuple let $q$ be the denominator of $\ML$. Theorem \ref{thm:UC} gives $\max_iv_i<3q$ on $\UC$, and the ratio is below $1$ on $\UA$ and $\UB$; among all near-tight sextuples with speeds at most $110$ the largest ratio is $105/43$, and the largest ratio among the sporadic ones is $46/25$, for $(1,3,4,5,18,46)$; note that here $q=25$ although the maximum is attained only at times with denominator $50$ (Remark \ref{rem:denominators}). Does $\max_iv_i<3q$ hold for every near-tight sextuple? More generally, Kravitz has pointed out that a bound $\max_iv_i\le C(n)\,q$ for all but finitely many near-tight $n$-tuples follows from \cite{JainKravitz} whenever the Lonely Runner Conjecture holds for $n-1$ and $n-2$ speeds, the constant being a maximum over the finitely many critical subtori; the homogeneity argument in the proof of Theorem \ref{thm:UC} is how such a constant is computed. Its value for other $n$, and whether the sporadic tuples respect it, are open.
\end{question}

\begin{question}
An effective version of Lemma \ref{lem:reduction} would make Theorems \ref{thm:six} and \ref{thm:five} unconditional statements about all near-tight tuples. The constants in \cite{GiriKravitz} are explicit but enormous; is there a direct argument bounding the speeds of a near-tight $n$-tuple that does not lie on a critical subtorus?
\end{question}

\section{Code and reproducibility}\label{sec:code}

The code is archived at Zenodo \cite{Code}, under the MIT license; \texttt{run\_all.sh} there reproduces every number quoted in this paper in about twenty minutes on two cores. The programs use only the Python standard library (the module \texttt{fractions} for exact arithmetic) except for an optional use of \texttt{numpy} in the witness filter of Theorem \ref{thm:three}, and the C programs need a C99 compiler. The runs reported here used Python 3.11 and gcc 13.

\begin{itemize}\sloppy
\item \texttt{lr\_ml.py}, \texttt{lr\_subtorus.py}, \texttt{lr\_classify.py}, \texttt{lr\_lemma\_s5.py} and \texttt{lr\_lemma\_s6.py}, driven by \texttt{reproduce.py}: the classification of Theorem \ref{thm:three} (about one minute), including the exact decision procedure for $\ML(U)>1/6$, the branch-and-prune verification of Lemma \ref{lem:onetwo} and the sixteen-case check of Lemma \ref{lem:plucker}. \texttt{lr5.py}: the same for Theorem \ref{thm:two}.
\item \texttt{lrk.h}: the exact kernel for $\ML$ in C (Lemma \ref{lem:pairtimes}), used by \texttt{fastrunner.c} and \texttt{fastrunner5.c} (Theorem \ref{thm:fast}), \texttt{prejump.c} (a numerical check of Proposition \ref{prop:prejump}), \texttt{ucscan.c}, \texttt{ucdump.c} and \texttt{ucpoints.c} (direct computations on $\UC$), and \texttt{survey.c} (the exhaustive searches of Section \ref{sec:cons}, with the membership tests for the critical subtori).
\item \texttt{jk\_uc.py}, \texttt{jk\_uc2.py}, \texttt{jk\_uc3.py} and \texttt{jk\_uc4.py}, driven by \texttt{uc\_theorem.py}: the analysis of Section \ref{sec:UC}. The first module builds the components and the functions $f_\ell$ as exact piecewise linear functions and computes the zones and thresholds; the second establishes the period and validates \eqref{eq:far}; the third performs the cone decomposition and the residue analysis; the fourth treats the lines and calls \texttt{ucpoints} for the direct checks. The driver prints every number quoted in Section \ref{sec:UC} in about two minutes.
\end{itemize}
The two implementations of the $\ML$ kernel, in Python and in C, were compared on several hundred random tuples with up to seven speeds, with agreement on the value and on the optimal times.

\section*{Acknowledgements}
I am grateful to Noah Kravitz for helpful conversations and for his feedback on earlier drafts of this paper, and to Cemal Payzin for pointing out an error in the value reported in Question \ref{q:ratio} of the first arXiv version. The responsibility for any error is mine alone.

The author used the language models Claude Opus 5 and Claude Fable 5.1 (Anthropic) as assistants during this work: for writing and cross-checking the verification code, for searching the literature, and for drafting and editing the text. All statements were checked by the author, who takes full responsibility for them.

\end{document}